\documentclass[a4paper,reqno,12pt]{amsart}

\usepackage{amsmath, amsfonts, amssymb, amsthm, amscd}
\usepackage{graphicx}
\usepackage{psfrag}
\usepackage{perpage}
\usepackage{url}
\usepackage{color}

\usepackage[utf8]{inputenc}
\usepackage[T1]{fontenc}

\usepackage[a4paper,scale={0.72,0.74},marginratio={1:1},footskip=7mm,headsep=10mm]{geometry}

\newcommand{\ind}{\mathbf{1}}

\newtheorem{theorem}{Theorem}

\newtheorem{corollary}[theorem]{Corollary}
\newtheorem{proposition}[theorem]{Proposition}

\def\p{\mathbb{P}}

\def\ind{\mbox{\rm 1\hspace{-0.04in}I}}

\newcommand{\formula}[2][nolabel]
{\ifthenelse{\equal{#1}{nolabel}}
 {\begin{align*} #2 \end{align*}}
 {\ifthenelse{\equal{#1}{}}
  {\begin{align} #2 \end{align}}
  {\begin{align} \label{#1} #2 \end{align}}
 }
}

\numberwithin{equation}{section}

\title[Exponentiality of first passage times]{Exponentiality of first passage times of continuous time Markov chains}

\author{Romain Bourget \and Lo\"ic Chaumont \and Natalia Sapoukhina}
\address{R. Bourget$^{1,2,3,4}$ -- \rm bourget@math.univ-angers.fr}
\address{L. Chaumont$^4$ -- \rm loic.chaumont@univ-angers.fr}
\address{N. Sapoukhina$^{1,2,3}$ -- \rm natalia.sapoukhina@angers.inra.fr}

\address{ $^1$ INRA, UMR1345 Institut de Recherche en Horticulture et Semences -- IRHS, SFR 4207, 
PRES UNAM, 42 rue Georges Morel, F-49071 Beaucouz\'e Cedex, France}
\address{$^2$ AgroCampus-Ouest, UMR1345 Institut de Recherche en Horticulture et Semences -- IRHS, F-49045 Angers, France}
\address{$^3$ Universit\'e d'Angers, UMR1345 Institut de Recherche en Horticulture et Semences -- IRHS, F-49045 Angers, France}
\address{$^4$ LAREMA UMR CNRS 6093, Universit\'e d'Angers, 2, Bd Lavoisier \\
Angers Cedex 01, 49045, France}


\thanks{This work was supported by MODEMAVE research project from the R\'egion
Pays de la Loire and by Department of "Sant\'e des Plantes et Environnement", INRA}

\date{\today}

\begin{document}

%\vspace*{1.7in}

\begin{abstract} Let $(X,\p_x)$ be a continuous time Markov chain with finite or countable state space $S$ and let $T$ be
its first passage time in a subset $D$ of $S$. It is well known that if $\mu$ is a quasi-stationary distribution relative
to $T$, then this time is exponentially distributed under $\p_\mu$. However, quasi-stationarity is not a necessary condition.
In this paper, we determine more general conditions on an initial distribution $\mu$ for $T$ to be exponentially
distributed under $\p_\mu$. We show in addition how quasi-stationary distributions can be expressed in terms of
any initial law which makes the distribution of $T$ exponential. We also study two examples in branching processes where
exponentiality does imply quasi-stationarity.
\end{abstract}

\maketitle

\section{Introduction}\label{int}

Let us denote by $P(t)=\{p_{ij}(t):i,j\in S\}$, $t\ge0$ the transition probability of a continuous time irreducible Markov chain
$X=\{(X_t)_{t\ge0},(\p_i)_{i\in S}\}$, with finite or countable state space $S$ and let $Q=\{q_{ij}:i,j\in S\}$ be
the associated $q$-matrix, that is $q_{ij}=p'_{ij}(0)$. We assume that $Q$ is conservative, that is $\sum_{j\in S}q_{ij}=0$,
for all $j\in S$, and that $X$ is not explosive. The transition probability (that will also be called the transition semigroup)
of $X$ satisfies the backward Kolmogorov's equation:
\begin{equation}\label{backward}
\frac{d}{dt}p_{ij}(t)=\sum_{k\in S}q_{ik}p_{kj}(t)\,.
\end{equation}
Let $D\subset S$ be some domain and define the first passage time by $X$ in $D$ by,
\begin{equation}\label{3178}
T=\inf\{t\ge0:X_t \in D\}\,.
\end{equation}
This work aims at characterizing  probability measures $\mu$ on $E=S\setminus D$ such that under $\p_{\mu}$,
the time  $T$ is exponentially distributed, that is, there exists $\alpha>0$, such that:
\begin{equation}\label{exponential}
\p_{\mu}(T>t)=e^{- \alpha t}\,.
\end{equation}
It is well known that when $\mu$ is a quasi-stationary distribution with respect to $T$, that is if
\begin{equation}\label{quasistat}
\p_{\mu}(X_t=i\,|\,T>t)=\mu_i\,,\;\;\;\mbox{for all $i\in E$ and $t\ge0$,}
\end{equation} then (\ref{exponential}), for some value $\alpha>0$, follows from a simple application of the Markov property,
see \cite{la} or \cite{ma} for example. Quasi-stationarity of $\mu$ holds if and only if $\mu$ is a left eigenvector
of the $q$-matrix of the process $X$ killed at time $T$, associated to the eigenvalue $-\alpha$, see \cite{np}.
However, quasi-stationarity is not necessary to obtain (\ref{exponential}). Some examples of non quasi-stationary distribution $\mu$
such that (\ref{exponential}) holds are given later on in this paper.\\

Our work was first motivated by population dynamics, where it is often crucial to determine the extinction time of a population or
the emergence time of a new mutant, see \cite{na,imn,cm,bcs2} for example. In many situations, those times
can be represented as first passage times of Markov processes in some particular domain. Then it is often much easier to find an initial
distribution, under which this first passage time is exponentially distributed than to compute its distribution under
any initial conditions.\\

Let us be more specific about applications to emergence times in biology which is the central preoccupation of the authors
in \cite{bcs2}. Adaptation to a new environment occurs by the emergence of new mutants. In adaptation theory, emergence can be described by
the estimation of the fixation time of an allele in the population. We may also imagine a parasite infecting a resistant or new host,
a pathogen evading chemical treatment, a cancer cell escaping from chemotherapy, etc. \cite{imn,ko,sh,hla}. An interesting and important
point is to estimate the law of the time at which these new mutant individuals emerge in the population, for example to estimate the
durability or the success probability of a new treatment or a new resistance. The emergence problem has already been considered in the
setting of branching processes \cite{la2,se,sh,acl}, for multitype Moran models in \cite{sc,dss,dm}, and for competition processes, in \cite{bcs2}.
In order to explain the latter case in more detail, let us recall that a competition process is a continuous time Markov chain $X=(X^{(1)},\dots,X^{(d)})$ 
with state space $S=\mathbb{N}^d$, for $d\ge2$, whose transition probabilities only allow jumps to certain nearest neighbors. Competition processes 
were introduced by Reuter \cite{re} as the natural extensions of birth and death processes and are often 
involved in epidemic models \cite{cl,hla,ko}. In \cite{bcs2}, the authors were interested in some estimation of the law of the first passage 
time $T$, when an individual of type $r$, $1\le r\le d$, first emerges from the population, that is
\[T=\inf\{t\ge0:X_t^{(r)}=1\}\,.\]
Then varying the birth, mutation, migration and death rates, some simulations of the law of the time $T$
allowed us to conclude that the consideration of interactions among two stochastic evolutionary forces,
mutation and migration, can expand our understanding of the adaptation process at the population
level. In particular, it showed under which conditions on mutation and migration rates, the pathogen can adapt swiftly to a
given multicomponent treatment.\\

This paper is organized as follows. In section 2, we establish a general criterion for a measure $\mu$ to satisfy (\ref{exponential})
and we study the connections between such measures and quasi-stationary or quasi-limiting distributions. Then, in the third section,
we give some sufficient conditions for (\ref{exponential}) involving the special structure of the chain on a partition of the state
space $E$.  In particular, Theorem \ref{th3} and its consequences allow us to provide some examples where exponentiality may hold without
quasi-stationarity. An example of application in adaptation theory is provided in Subsection \ref{ept1}. The fourth section is devoted to the 
presentation of some examples in the setting of branching processes where exponentiality implies quasi-stationarity.

\section{From exponentiality to quasi stationarity}\label{section2}

\noindent We first introduce the killed process at time $T$, as follows:
\begin{equation}\label{killedproc}
X^T_t=\left\{\begin{array}{ll}
X_t\,,\;\;\mbox{if $t<T$,}\\
\Delta\,,\;\;\mbox{if $t\ge T$,}
\end{array}\right.
\end{equation}
where $\Delta$ is a cemetery point. Then $X^T$ is a continuous time Markov chain which is valued in
$E_\Delta:=E\cup\{\Delta\}$. Moreover if we define the killing rate by
\begin{equation}\label{2035}
\eta_i=\sum_{j\in D}q_{ij}\,,
\end{equation}
then the $q$-matrix $Q^T=(q^T_{ij})$ of $X^T$ is given by
\begin{equation}\label{qmatrix}
q^T_{ij}=\left\{\begin{array}{ll}
q_{ij}\,,\;\;\;i,j\in E\\
q_{i\Delta}=\eta_i\,,\;\;\;i\in E\\
q_{\Delta j}=0\,,\;\;\;j\in E_{\Delta}\,.\end{array}
\right.
\end{equation}
From our assumptions, $Q^T$ is obviously conservative and $X^T$ is non explosive. In particular,
$Q^T$ is the $q$-matrix of a unique transition probability that
we will denote by $P^T(t)=(p_{ij}^T(t))_{i,j\in E_\Delta}$, $t\ge0$, and which is expressed as
\begin{equation}
p_{ij}^T(t)=\left\{\begin{array}{lll}
\p_{i}(X_t=j,t<T)\,,&\mbox{if $i,j\in E$,}\\
\p_{i}(t\ge T)\,,&\mbox{if $i\in E$ and $j=\Delta$,}\\
1_{j=\Delta}\,,&\mbox{if $i=\Delta$ and $j\in E_\Delta$.}
\end{array}\right.
\end{equation}
Then this semigroup inherits the Kolmogorov backward equation from (\ref{backward}):
\begin{equation}\label{1635}
\frac{d}{dt}p_{ij}^T(t)=\sum_{k\in E_\Delta}q^T_{ik}p_{kj}^T(t)\,.
\end{equation}
Henceforth, all  distributions $\nu$ on $E_\Delta$ that will be considered will not charge the state $\Delta$,
i.e. $\nu_\Delta=0$.
In this section, we shall often consider initial distributions $\mu=(\mu_i)_{i\in E_\Delta}$ for $(X_t^T)$, on $E_\Delta$
satisfying the following differentiability condition:
\begin{equation}\label{8135}
\mbox{$\mu P^T(t)$ is differentiable and $\displaystyle\frac d{dt}\mu P^T(t)=\mu\frac d{dt}P^T(t)$, $t>0$}\,.
\end{equation}
We extend the family of probabilities $(\p_i)_{i\in E}$ to $i=\Delta$, in accordance with the definition of $(P^T(t))$
and for each $t\ge0$, we define the probability distribution $\mu(t)$ on $E_\Delta$ as follows:
\begin{equation}\label{2147}
\mu_i(t)=\mathbb{P}_{\mu}(X_t^T=i\,|\,T>t)\,,\;\;i\in E_\Delta\,.
\end{equation}
We define the vector $\delta$ by $\delta_i=0$, if $i\in E$ and $\delta_\Delta=1$.

\begin{theorem}\label{main}
Let $\mu$ be a distribution on $E_\Delta$.
\begin{itemize}
\item[$(i)$] Assume that $\mu$ satisfies condition $(\ref{8135})$,
then there is $\alpha>0$ such that $\mathbb{P}_{\mu}(T>t)=e^{-\alpha t}$, for all $t\ge0$ if and only if
\begin{equation}\label{eqadd3}
\mbox{$\mu(t)$ is differentiable and $\displaystyle\mu'(t)=e^{\alpha t} (\mu Q^T+ \alpha(\mu-\delta))P^T(t)$, $t>0$.}
\end{equation}
\item[$(ii)$] Assume that there is $\alpha>0$ such that $\mathbb{P}_{\mu}(T>t)=e^{-\alpha t}$, for all $t\ge0$, then conditions
$(\ref{8135})$ and $(\ref{eqadd3})$  are equivalent.
\item[$(iii)$]  When $(\ref{eqadd3})$ is satisfied, the rate $\alpha$ may be expressed as
\begin{equation}\label{alpha}\alpha=\sum_{i\in E}\eta_i\mu_i\,.
\end{equation}
\end{itemize}
\end{theorem}
\begin{proof}
Note that the condition $\mathbb{P}_{\mu}(T>t)=e^{-\alpha t}$ is equivalent to
$\mathbb{P}_{\mu}(X_t^T=i,t<T)=e^{-\alpha t}\mu_i(t)$. Therefore, since
\[\mathbb{P}_{\mu}(X^T_t=i)=\p_{\mu}(X_t^T=i,t<T)+\ind_{i=\Delta}\p_{\mu}(t\ge T)\,,\]
the transition function $P^T(t)$ of $X^T$ satisfies
\begin{equation}\label{eqadd1}
\mu P^T(t)=e^{-\alpha t}\mu(t)+(1-e^{-\alpha t})\delta\,.
\end{equation}
Then from the differentiability condition (\ref{8135}), we see that $\mu(t)$ is differentiable and from the Kolmogorov backward
equation (\ref{1635}), we obtain
\begin{eqnarray}
\mu\frac d{dt}P^T(t)&=&-\alpha e^{-\alpha t}\mu(t)+ e^{-\alpha t} \mu'(t)+\alpha e^{-\alpha t}\delta\nonumber\\
&=&\mu Q^TP^T(t)\,.\label{4925}
\end{eqnarray}
Then from (\ref{eqadd1}), we have $e^{-\alpha t} \mu(t) = \mu P^T(t)-(1-e^{- \alpha t})\delta$ and since $\delta P^T(t)=\delta$,
for all $t\ge 0$, we see that equation (\ref{4925}) may be expressed as
\[\mu'(t)=e^{\alpha t} (\mu Q^T+ \alpha(\mu-\delta))P^T(t)\,,\;\;\;t\ge0\,.\]
Conversely, if condition (\ref{eqadd3}) is satisfied, then from (\ref{8135}), we can write equation (\ref{4925}).
Integrating this expression, we get (\ref{eqadd1}) which implies that $\mathbb{P}_{\mu}(T>t)=e^{-\alpha t}$, for all $t\ge0$.
The first assertion of the theorem is proved.

Now if $\mathbb{P}_{\mu}(T>t)=e^{-\alpha t}$, for all $t\ge0$, then we have (\ref{eqadd1}), so that if condition $(\ref{8135})$ is
satisfied, then $\mu(t)$ is differentiable and
\begin{equation}\label{5821}
\frac d{dt}\mu P^T(t)=-\alpha e^{-\alpha t}\mu(t)+ e^{-\alpha t} \mu'(t)+\alpha e^{-\alpha t}\delta\,.
\end{equation}
Moreover from the Kolmogorov backward equation and (\ref{5821}), we have
$\mu Q^TP^T(t)=-\alpha e^{-\alpha t}\mu(t)+ e^{-\alpha t} \mu'(t)+\alpha e^{-\alpha t}\delta$, which is $(\ref{eqadd3})$. The
converse is easily derived from similar arguments, so the second assertion is proved.

Then from equation (\ref{eqadd3}), we obtain
\begin{equation}\label{5478}
\lim_{t\rightarrow0}\mu'(t)= (\mu Q^T+ \alpha(\mu-\delta))P^T(0)\,.
\end{equation}
On the other hand, note that $\mu_\Delta(t)=0$, for all $t\ge0$, so that in particular $\mu_\Delta=\mu_\Delta(0)=0$ and
$\lim_{t\rightarrow0}\mu'_\Delta(t):=\mu_\Delta'(0)=0$. Finally, taking equality (\ref{5478}) at $\Delta$ yields
\[\mu Q^T_\Delta=\sum_{i\in E_\Delta} \mu_i q^T_{i\Delta} = \sum_{i\in E} \mu_i \eta_i= \mu'_\Delta(0) -
\alpha(\mu_\Delta-\delta_\Delta)=\alpha\,,\]
which proves the third assertion of the theorem.\\
\end{proof}
\noindent Note that the equality in (\ref{eqadd3}), once restricted to the set $E$ can be simplified as  
$\displaystyle\mu'(t)=e^{\alpha t}\mu(Q^T+ \alpha I)P^T(t)$,  which highlights the importance of the operator $Q^T+\alpha I$. 
This also applies to the next results.\\

\noindent {\bf Remarks.} {\it $1.$ It is important to note that a distribution $\mu$ on $E_\Delta$ may satisfy
$\p_\mu(T>t)=e^{-\alpha t}$, $t\ge0$, whereas $(\ref{8135})$ does not hold. Examples are given in the remark 
after Theorem $\ref{th3}$.

\noindent $2.$ When $E$ is finite, condition $(\ref{8135})$ is clearly satisfied. In the infinite case,
this condition may appear theoretical to some extend and sometimes difficult to check when not much is known on the transition
probability. However it is possible to obtain quite simple conditions implying $(\ref{8135})$. For instance, observe that from
$(\ref{1635})$, for all $i,j\in E_\Delta$ and $t>0$,
\begin{eqnarray}
\left|\frac{d}{dt}p_{ij}^T(t)\right|&\le&\sum_{k\in E_\Delta}\left|q^T_{ik}p_{kj}^T(t)\right|\nonumber\\
&\le&\sum_{k\in E_\Delta}|q^T_{ik}|\nonumber\\
&=&-2q^T_{ii}\,.\label{1755}
\end{eqnarray}
A sufficient condition for $(\ref{8135})$ to hold is then
\begin{equation}\label{8105}
\sum_{i\in E}q_i\mu_i<\infty\,,
\end{equation}
where $q_i=-q_{ii}^T$. The latter condition is satisfied in particular when the $q_i$'s are bounded.}\\

Recall definition (\ref{quasistat}) of quasi-stationarity. In our setting, it is equivalent to the following statement:
a distribution $\mu$ on $E_\Delta$, is quasi-stationary if
\begin{equation}\label{6967}
\mu_i=\mu_i(t)\,,\;\;\;\mbox{for all $t\ge0$ and $i\in E_\Delta$.}
\end{equation}
We will simply say that $\mu$ is a quasi-stationary distribution. Then,
let us state the following classical result, already mentioned in the introduction.
\begin{theorem}[\cite{pvj}]\label{pvj}
A distribution $\mu$ on $E_\Delta$ is quasi-stationary if and only if the equation
\begin{equation}\label{6457}
\mu Q^T=-\alpha \mu + \alpha\delta \,,
\end{equation}
holds for some $\alpha>0$.
$($Note that $(\ref{6457})$ is equivalent to $\mu Q^T_i=-\alpha \mu_i$, for all $i\in E$.$)$
\end{theorem}

\noindent In \cite{pvj} it is proved that (\ref{6457}) is equivalent to the fact that $P^T$ satisfies the 
Kolmogorov forward equation, which is the case under our assumptions, that is
\begin{equation}\label{1735}
\frac{d}{dt}p_{ij}^T(t)=\sum_{k\in E}p_{ik}^T(t)q^T_{kj}\,.
\end{equation}
Knowing condition (\ref{8135}), Theorem \ref{pvj} easily follows from an application of the 
Kolmogorov backward equation. Actually, under this assumption Theorem \ref{pvj} can be derived from  Theorem \ref{main}. 
As a consequence of both these theorems we also obtain that  $(\ref{8135})$ holds whenever $\mu$  is quasi-stationary.
\begin{corollary}\label{1832}
If $\mu$ is a quasi-stationary distribution then condition $(\ref{8135})$ holds.
\end{corollary}
\begin{proof} If $\mu$ is quasi-stationary, then it follows from  (\ref{6967}) and the Markov property that
$\mathbb{P}_{\mu}(T>t)=e^{-\alpha t}$, for some $\alpha>0$ (this fact is well known, see \cite{ma}, for instance). Moreover
the function $\mu(t)$ is differentiable and $\mu'(t)=0$, for all $t\ge0$. On the other hand, from Theorem \ref{pvj},
equation (\ref{6457}) holds. Therefore, condition (\ref{eqadd3})
holds, so that (\ref{8135}) is satisfied from part $(ii)$ of Theorem \ref{main}.
\end{proof}
A distribution $\pi$ on $E_\Delta$ is called the quasi-limiting distribution (or the Yaglom limit) of a distribution
$\mu$ on $E_\Delta$, if it satisfies
\begin{equation}\label{5644}
\lim_{t\rightarrow\infty}\p_{\mu}(X_t^T=i\,|\,T>t)=\pi_i\,,\;\;\mbox{for all $\;i\in E_\Delta$.}
\end{equation}
Then a well known result asserts that any quasi-limiting distribution is also a quasi-stationary distribution,
see for example \cite{la,ma,me}. Recall also that if $\pi$ is the quasi-limiting distribution of some distribution $\mu$,
then the rate $\alpha$ satisfying (\ref{exponential}) is given by the expression
\begin{equation}\label{5614}
\alpha=\inf\left\{a\ge0:\int_0^\infty e^{at}P_{i}(T>t)\,dt=\infty\right\}>0\,,
\end{equation}
which does not depend on the state $i\in E$, see Section 3 in \cite{jr} for instance. 
As an application of Theorem \ref{main} and the above remarks, we show in the next corollary how to construct quasi-stationary 
distributions from distributions satisfying (\ref{exponential}).

\begin{corollary}\label{coadd2}
Let $\mu$ be a distribution on $E_\Delta$ such that $\p_\mu(T>t)=e^{-\alpha t}$, $t\ge0$,
for some $\alpha>0$ and satisfying $(\ref{8135})$. If $\mu$ admits a quasi-limiting distribution, $\pi$, then the latter is
given by:
\[\pi=\mu+\int_0^{\infty} (\mu Q^T+ \alpha(\mu-\delta)) e^{\alpha t} P^T(t)\,dt\,,\]
where $\int_0^{\infty} (\mu Q^T+ \alpha(\mu-\delta)) e^{\alpha t} P^T(t)\,dt$ should be understood as a possibly improper
integral. In particular, $\pi$ is a quasi-stationary distribution on $E_\Delta$.
\end{corollary}
\begin{proof}
Under these assumptions, it follows from Theorem \ref{main} that for all $t\ge0$, $\mu'(t)=e^{\alpha t}
(\mu Q^T+\alpha(\mu-\delta))P^T(t)$. Moreover, since $\p_{\mu}(T>0)=1$, $\mu(t)$ is continuous at 0 and $\mu(0)=\mu$,
so that
\[\mu(t)-\mu=\int_0^t (\mu Q^T+ \alpha(\mu-\delta))  e^{\alpha u} P^T(u)du\,.\]
Since $\mu(t)$ converges to a proper distribution $\mu$, as $t$ tends to $\infty$, it follows that the improper integral
$\int_0^\infty (\mu Q^T+ \alpha(\mu-\delta)) e^{\alpha u} P^T(u)du=\lim_{t\rightarrow+\infty}
\int_0^t (\mu Q^T+ \alpha(\mu-\delta)) e^{\alpha u} P^T(u)du$ exists and is finite. The fact that $\pi$ is a
quasi-stationary distribution follows from the results which are recalled before the statement of the corollary.
\end{proof}
\noindent Corollary \ref{coadd2} may be interpreted as follows: if $\mu$ is such that $T$ is exponentially distributed under
$\p_\mu$ and admits a Yaglom limit, then the correction term which allows us to obtain a quasi-stationary distribution from
$\mu$ is $\int_0^{\infty} (\mu Q^T+ \alpha(\mu-\delta)) e^{\alpha t} P^T(t)\,dt$.\\

The next results of this section show that whenever there exists a non quasi-stationary distribution which makes the
time $T$ exponentially distributed, then under some conditions, we may construct a whole family of distributions having
the same property.

\begin{proposition}\label{fam} Let $\mu$ be a distribution on $E_\Delta$ satisfying $(\ref{8135})$ and such that
$\p_{\mu}(T>t)=e^{-\alpha t}$, $t\ge0$, for some $\alpha>0$. Let us define the vector
$(\mu_{i}^{(1)})_{i\in E_\Delta}$, by
\begin{equation}\label{9445}
\mu_j^{(1)}=\frac{-1}{\alpha} \sum \limits_{i\in E_\Delta} \mu_iq^{T}_{ij}\,,\;\;\;j\in E\,,
\;\;\;\mu_\Delta^{(1)}=0\,.
\end{equation}
If for all $j\in E$,
\begin{equation}\label{9345}
0\le -\sum \limits_{i\in E_\Delta} \mu_iq^{T}_{ij}\le \alpha\,,
\end{equation}
then $(\mu_{i}^{(1)})_{i\in E_\Delta}$ is a distribution  on $E_\Delta$ which satisfies $\p_{\mu^{(1)}}(T>t)=e^{-\alpha t}$,
for all $t\ge0$.
\end{proposition}
\begin{proof} The assumption $\p_{\mu}(T>t)=e^{-\alpha t}$ is equivalent to
\begin{equation}\label{expo}
\sum \limits_{i \in E} \mu_ip^T_{i\Delta}(t)=1-e^{-\alpha t}\,.
\end{equation}
Using condition (\ref{8135}) and the Kolmogorov backward equation (\ref{1635}),
we obtain by differentiating the latter equality
\begin{eqnarray*}
\sum_{i\in E} \left( \sum_{j \in E_\Delta} q^T_{ij} p^T_{j\Delta}(t)\right)\mu_i &=&\alpha e^{-\alpha t}\,.
\end{eqnarray*}
Decomposing the left hand side and using (\ref{qmatrix}) and (\ref{alpha}), we obtain
\begin{eqnarray*}
\sum_{i\in E} \left( \sum_{j \in E_\Delta} q^T_{ij} p^T_{j\Delta}(t)\right)\mu_i &=&
\sum_{i\in E} \left( q^T_{i\Delta}+\sum_{j \in E} q^T_{ij} p^T_{j\Delta}(t)\right)\mu_i \\
&=&\alpha+\sum_{i,j \in E} q^T_{ij} p^T_{j\Delta}(t)\mu_i \\
&=&\alpha e^{-\alpha t}\,,
\end{eqnarray*}
which gives
\begin{equation}\label{4579}
\sum_{j\in E} p_{j\Delta}^T(t) \left(\frac{-1}{\alpha}\sum_{i\in E}\mu_iq^T_{ij}\right)=1-e^{-\alpha t}\,.
\end{equation}
Then from condition (\ref{9345}), we may let $t$ tend to $\infty$ in (\ref{4579}), in order to obtain by monotone
convergence that $\sum_{j\in E}\mu^{(1)}_j=1$, so that
\[\mu_j^{(1)}=\frac{-1}{\alpha}\sum_{i\in E}\mu_iq^T_{ij}\,,\;\;j\in E\,,\;\;\;\mu_\Delta^{(1)}=0\]
is a distribution on $E_\Delta$. Moreover
(\ref{4579}) is equation (\ref{expo}) where we have replaced $\mu$ by $\mu^{(1)}$, so that $\mu^{(1)}$ satisfies
$\p_{\mu^{(1)}}(T>t)=e^{-\alpha t}$.
\end{proof}

\begin{corollary}\label{2362} Let $\mu$ be a distribution on $E_\Delta$ and $\alpha>0$.  For $n\ge1$, let us denote by 
$q^{n,T}_{ij}$ the entries of $(Q^T)^n$ and define the vector $(\mu_{i}^{(n)})_{i\in E_\Delta}$, by
\begin{equation}\label{9445}
\mu_j^{(n)}=\frac{(-1)^n}{\alpha^n} \sum \limits_{i\in E_\Delta} \mu_iq^{n,T}_{ij}\,,\;\;\;j\in E\,,
\;\;\;\mu_\Delta^{(n)}=0\,.
\end{equation} 
Then,
\begin{itemize}
\item[$1.$]  $\mu^{(n)}$ is a quasi-stationary distribution associated to the rate $\alpha$, for some $n\ge1$, if and only if 
$\mu^{(k)}=\mu^{(k+1)}$, for all $k\ge n$.
\item[$2.$] Assume that for all $j\in E$, $\sum_{i\in E}q_{ij}^T<\infty$. 
If the sequence $(\mu^{(n)})$ converges, as $n \rightarrow \infty$, 
toward a proper distribution $\mu^{(\infty)}$, then $\mu^{(\infty)}$ is a quasi-stationary distribution.
\item[$3.$] Assume that $E$ is finite. If $\p_\mu(T>t)=e^{-\alpha t}$, $t\ge0$ and if for all 
$n\ge1$, 
\begin{equation}\label{5422}
0\le (-1)^n\sum \limits_{i\in E_\Delta} \mu_iq^{n,T}_{ij}\le \alpha^n\,,
\end{equation}  
then for all $n\ge1$, $(\mu_{i}^{(n)})_{i\in E_\Delta}$ is a distribution  on $E_\Delta$ which satisfies 
$\p_{\mu^{(n)}}(T>t)=e^{-\alpha t}$, for all $t\ge0$.
\end{itemize}
\end{corollary}
\begin{proof} The proof of the first part simply follows from the identity:
\begin{equation}\label{2345}
\mu^{(k+1)}_i= \frac{-1}{\alpha}\mu^{(k)}Q^T_i\,,\;\;\;i\in E\,,
\end{equation}
and Theorem \ref{pvj}.

The second assertion is a consequence of the same observation, which leads, by passing to the limit thanks to the assumptions to,
$\mu^{(\infty)}Q^T_i=-\alpha\mu^{(\infty)}_i$, $i\in E$. Then we conclude by applying Theorem \ref{pvj}.

Then the third part follows from Proposition \ref{fam} by induction. Indeed, first recall that since $E$ is finite, condition $(\ref{8135})$ is 
satisfied for any distribution. If the result is true for $\nu:=\mu^{(n)}$, then from the inequality 
$0\le (-1)^{n+1}\sum \limits_{i\in E_\Delta} \mu_iq^{n+1,T}_{ij}\le \alpha^{n+1}$ and identity (\ref{2345}), we derive that for all $j\in E$,
\[0\le -\sum \limits_{i\in E_\Delta} \nu_iq^{T}_{ij}\le \alpha\,,\]
so that from Proposition \ref{fam},  $\nu_j^{(1)}:=\frac{-1}{\alpha} \sum \limits_{i\in E_\Delta} \nu_iq^{T}_{ij}=\mu^{(n+1)}$ is a distribution  
on $E_\Delta$ which satisfies  $\p_{\nu^{(1)}}(T>t)=e^{-\alpha t}$, for all $t\ge0$.
\end{proof}
\noindent As we have already observed, if $\sup_{i\in E}q_i\le\alpha$, where
$q_i:=-q_{ii}^T$, then condition $(\ref{8135})$ is satisfied, but also for all $j\in E$,
\begin{equation}\label{8845}
(-1)^n\sum \limits_{i\in E_\Delta} \mu_iq^{n,T}_{ij}\le \alpha^n\,,
\end{equation}
which provides the second inequality in (\ref{5422}).  An interesting problem is then to determine simple conditions ensuring
the first inequality in (\ref{5422}), that is nonnegativity of the term $(-1)^n\sum \limits_{i\in E_\Delta} \mu_iq^{n,T}_{ij}$.

\section{Sufficient conditions for exponentiality.}\label{ept1}

Let us keep the notation of the previous sections. The next theorem provides sufficient conditions for a distribution $\mu$
to insure that $T$ is exponentially distributed under $\p_\mu$. As shown in Section \ref{exemples}, this result allows us to construct examples 
for which such distributions exist.
\begin{theorem}\label{th3}
Let $\{E_1,E_2,\dots\}$ be a finite or infinite partition of $S$ containing at least two elements and with $E_1=D$
$($in particular $\{E_2,E_3,\dots\}$ is a partition  of $E$$)$. Assume that:
%$\mu$ is a distribution with support in $E$ that satisfies the following condition:
\begin{itemize}
\item[$(i)$] For all $k\ge2$ and $l\ge1$ and for all $i\in E_k$, the quantity $\sum_{j\in E_l} q_{ij}$ does not depend on $i$. For $i\in E_k$, 
we set
\begin{equation}\label{3741}
\bar{q}_{kl}:=\sum_{j\in E_l} q_{ij}\,.
\end{equation}
\end{itemize}
Let $\mu$ is a distribution on $E_\Delta$,  with support in $E$. The following two conditions are equivalent.
\begin{itemize}
\item[$(ii)$]
For all $k\ge1$, the quantity $\mathbb{P}_{\mu}(X_t \in E_k \mid T>t)$ does not depend on $t\ge0$. More specifically,
we have,
\begin{equation}\label{3421}
\mathbb{P}_{\mu}(X_t \in E_k \mid T>t)=\bar{\mu}_k\,,\;\;\;t\ge0\,,
\end{equation}
where  $\bar{\mu}_k=\sum \limits_{i\in E_k} \mu_i$.
\item[$(iii)$] There exists $\alpha>0$, such that
\[\bar{\mu}\bar{Q}=-\alpha\bar{\mu}+\alpha{\rm d}\,,\]
where $\bar{Q}=(\bar{q}_{kl})_{k,l\ge1}$, $\bar{q}_{1k}=0$, for $k\ge1$,
 $\bar{q}_{kk}=-\sum_{l\ge1,\,l\neq k}\bar{q}_{kl}$, for $k\ge1$,
$\bar{\mu}=(\bar{\mu}_k)_{k\ge1}$ and ${\rm d}=(1,0,0,\dots)$.
\end{itemize}
Moreover, if conditions $(i)$ and $(ii)$ (or equivalently conditions $(i)$ and $(iii)$) are satisfied, then $T$ is
exponentially distributed under $\p_\mu$, with parameter $\alpha$  given by
\begin{equation}\label{5032}
\alpha=\sum_{k\ge1}\bar{q}_{k1}\bar{\mu}_k\,.
\end{equation}
\end{theorem}
\begin{proof} Let $(Y_t)_{t\ge0}$ be the continuous time process with values in $\mathbb{N}=\{1,2,\dots\}$ which is defined by
$Y_t=k$, if $X_t\in E_k$, that is
\[Y_t=\sum_{k\ge1}k\ind_{\{X_t\in E_k\}}\,,\;\;\;t\ge0\,.\]
Observe that $T=\inf\{t:Y_t=1\}$. Then under assumption $(i)$, the absorbed process
\begin{equation}\label{killedproc}
Y^T_t=\left\{\begin{array}{ll}
Y_t\,,\;\;\mbox{if $t<T$,}\\
1\,,\;\;\mbox{if $t\ge T$,}
\end{array}\right.
\end{equation}
is a continuous time Markov chain with $q$-matrix $\bar{Q}=(\bar{q}_{kl})_{k,l\ge1}$, as defined in $(iii)$.
See for instance Section 3.4 in \cite{ma}.

Then recall from $(iii)$, the definition of the measure $\bar{\mu}$ on $\mathbb{N}$: $\bar{\mu}_k=\sum_{i\in E_k}\mu_i$, $k\ge2$
and $\bar{\mu}_1=0$. Let $(\bar{\p}_k)_{k\ge1}$ be the family of probability laws associated to the Markov process $(Y_t)_{t\ge0}$.
For all $k\ge2$,
\begin{eqnarray}
\p_\mu(X_t\in E_k,t<T)&=&\sum_{i\in E}\mu_i\p_i(X_t\in E_k,t<T)\nonumber\\
&=&\sum_{l=2}^n\sum_{i\in E_l}\mu_i\p_i(X_t\in E_k,t<T)\nonumber\\
&=&\sum_{l=2}^n\bar{\mu}_l\bar{\p}_l(Y_t=k,t<T)\nonumber\\
&=&\bar{\p}_{\bar{\mu}}(Y_t=k,t<T)\,,\label{9373}
\end{eqnarray}
where the third equality follows from the fact that $\p_i(X_t\in E_k,t<T)=\bar{\p}_l(Y_t=k,t<T)$, for all $i\in E_l$.
Assume that condition $(ii)$ holds, then we derive from (\ref{3421}) and (\ref{9373}) that for all $k\ge2$,
\[\bar{\p}_{\bar{\mu}}(Y_t=k\,|\,t<T)=\bar{\mu}_k\,,\]
which means that $\bar{\mu}$ is a quasi stationary distribution with respect to the lifetime of the Markov process $Y^T$.
In particular, thanks to Theorem \ref{pvj}, $(ii)$ and $(iii)$ are equivalent. Moreover, since (\ref{eqadd3}) in
Theorem \ref{main} is satisfied, then from $(iii)$ in this theorem, $T$ is exponentially distributed under $\p_{\bar{\mu}}$,
with parameter $\alpha=\sum_{k=2}^n\bar{q}_{k1}\bar{\mu}_k$. We conclude from equality (\ref{9373}) which shows that
$\p_\mu(t<T)=\bar{\p}_{\bar{\mu}}(t<T)$.
\end{proof}

\noindent {\bf Remark.} {\it Let us focus on two very particular situations, where Theorem $\ref{th3}$ can be applied. First, in the particular 
case where the partition $\{E_2,E_3,\dots\}$ of $E$ is reduced to the singletons of $E$, then condition $(ii)$ is obviously satisfied and condition 
$(i)$ simply means that $\mu$ is quasi-stationary with respect to $T$, hence the conclusion  follows from Theorem $\ref{pvj}$. 

Then recall the definition of $\eta_i$ in $(\ref{2035})$. In contrast to the latter situation, by 
considering $\{E\}$ as a partition of $E$, it follows from Theorem $\ref{th3}$ that if there exists $\alpha>0$ such that $\eta_i=\alpha$, for all  $i\in E$, 
then the first passage time $T$ has an exponential distribution with parameter $\alpha$ under $\p_\mu$, for all initial distributions $\mu$  with support in 
$E$. This result follows also from direct arguments, see Proposition $\ref{prop1}$ below for instance. In the case where $S$ is finite, it is stated in 
Proposition $2.1$, $(ii)$ of \cite{ma}.  Note that, if in addition there is a Yaglom limit, $\mu$, as recalled in the previous section, then in this case, 
$\mu$ is explicitly given  on $E$ by
\[\mu=1_{\{i\}}+\int_0^\infty1_{\{i\}}(Q^T+\alpha I)P^T(t)e^{\alpha t}\,dt\,,\]
for all $i\in E$. In particular, this expression does not depend on $i$. Note that in this case, $\mu$ corresponds to the stationary distribution of the unkilled
process $X$. Finally, let us emphasize that from this particular situation, we can construct
examples where an initial distribution $\mu$ satisfies $(\ref{exponential})$ but not $(\ref{8135})$.}\\

Actually it is always possible to compare the distribution of $T$ with the exponential law, as Proposition \ref{prop1} shows.
It  provides exponential bounds for the distribution function of the first passage time.
\begin{proposition}\label{prop1} Define the rates $\alpha_0=\inf_{i\in E}\eta_i$ and $\alpha_1=\sup_{i\in E}\eta_i$,
where $\eta_i$ is defined in $(\ref{2035})$.
Then the tail distribution of the first passage time $T$ satisfies the inequalities:
\begin{equation}\label{encadr}
e^{-\alpha_1 t}\le \p_i(t<T)\le e^{-\alpha_0 t}\,,
\end{equation}
for all $t\ge0$ and for all $i\in E$.
\end{proposition}
\begin{proof} By definitions (\ref{2035}) and (\ref{qmatrix}), we obtain that $\alpha_0\le q_{k\Delta}\le \alpha_1$, for all $k$. From 
these inequalities and Kolmogorov's forward equation at state $i\in E$ and $\Delta$, i.e.
\[\frac{d}{dt}p^T_{i\Delta}(t)=\sum_{k\in E}p^T_{ik}(t)q_{k\Delta}\,,\]
we derive that,
\[\alpha_0 \p_i(t<T)\le \frac{d}{dt}\p_i(t<T)\le \alpha_1\p_i(t<T)\,.\]
The result follows immediately. 
\end{proof}

\section{Examples and application}\label{exemples}

\subsection{Two examples of exponentality} With the aim of illustrating the previous results, we 
provide in this subsection two examples of  non quasi-stationary distributions $\mu$ such that $T$ is exponentially distributed under $\p_\mu$.\\

\noindent {\it An  example when the state space is finite}: With the same notations as in Theorem \ref{th3}, let $S=\{1,2,3,4,5,6\}$, 
$E_1=D=\{1\}$, $E_2=\{2,3\}$, $E_3=\{4,5\}$, $E_4=\{6\}$ and let us define the $q$-matrix,
\[Q=\left(\begin{array}{cccccc}
-6&0&2&1&1&2\\
1&-8&1&1&2&3\\
1&0&-7&2&1&3\\
1&1&1&-8&2&3\\
1&0&2&1&-7&3\\
1&1&0&1&1&-4
\end{array}
\right)\,,\]
which clearly satisfies the general conditions of this paper, see Section \ref{int}, as well as condition 
$(i)$ of Theorem \ref{th3}.  Let $Q'$ be the $q$-matrix $Q$ (or equivalently $Q^T$), to which the 
first line and the first column have been removed and let $\bar{Q}'$ be the $q$-matrix $\bar{Q}$ to which the 
first line and the first column have been removed, that is
\[Q'=\left(\begin{array}{ccccc}
-8&1&1&2&3\\
0&-7&2&1&3\\
1&1&-8&2&3\\
0&2&1&-7&3\\
1&0&1&1&-4
\end{array}
\right)\;\;\;\mbox{and}\;\;\;\bar{Q}'=\left(\begin{array}{ccc}
-7&3&3\\
2&-6&3\\
1&2&-4
\end{array}
\right)\,.\]
The Perron-Frobenius eigenvalue of $\bar{Q}'$ is $\lambda=-1$ and the associated normalized left eigenvector is $\nu=(3/16,5/16,1/2)$. In 
particular, $\nu\bar{Q}'=-\nu$, so that $\bar{\mu}=(0,3/16,5/16,1/2)$ and $\bar{Q}$ satisfy condition $(iii)$ of Theorem \ref{th3} 
with $\alpha=1$.

Then Theorem \ref{th3} asserts that the initial distribution $\mu=(0,3/32,3/32,5/32,5/32,1/2)$ is such that under $\p_\mu$, the emergence time $T$ 
is exponentially distributed with parameter 1. Moreover, since $\mu Q^T=(0,-3/32,-3/32,-5/16,0,-1/2)$, then relation (\ref{6457}) in Theorem \ref{pvj} 
cannot be satisfied for any $\alpha>0$, and hence $\mu$ is not a quasi-stationary distribution. 

Also, note that $\mu$ satisfies conditions of Proposition \ref{fam}. Then with the notation of this proposition, we have
$\mu^{(1)}=(0,3/32,3/32,5/16,0,1/2)$, so that condition (\ref{9345}) is satisfied and from this proposition, $\mu^{(1)}$ is another distribution 
such that under $\p_{\mu^{(1)}}$, $T$ is exponentially distributed with parameter 1. Moreover, we can check as above that $\mu^{(1)}$ is not a 
quasi-stationary  distribution.\\

\noindent {\it Exponentiality in $\mathbb{Z}^d$-valued L\'evy processes.} In this example the state space is $S=\mathbb{Z}^d$, with $d\ge2$
and $X$ is a $d$-dimensional compound Poisson process. In particular, $X$ issued from 0, can be represented on some probability space 
$(\Omega,\mathcal{F},\p)$, as, 
\[X_t=\sum_{k=0}^{N_t}\xi_k\,,\;\;\;t\ge0\,,\]
where $(\xi_k)_{k\ge1}$ is a sequence of i.i.d.~random variables with distribution on $\mathbb{Z}^d\setminus\{0\}$, $\xi_0=0$ and 
$(N_t)_{t\ge0}$ is a standard Poisson process, that is independent of the sequence $(\xi_k)_{k\ge1}$. As usual, $\p_x$, $x\in\mathbb{Z}^d$ will 
denote the family of probability measures such that $X$ starts from $x$ under $\p_x$.

Then assume that there is a linear transformation $M:\mathbb{Z}^d\rightarrow \mathbb{Z}^d$, such that the coordinates of the compound 
Poisson process $K_t:=MX_t$, $t\ge0$, are independent and non degenerate. 
Assume moreover that for some vector $u=(u_1,\dots,u_d)\in\mathbb{Z}^d$, whose $d'$ coordinates ($0<d'<d$) are equal to 0,
the support of the distribution of ${}^tuM\xi_1$, is a set of the form $\{-a,-a+1,\dots,-1,1,\dots,b-1,b\}$, 
for some $0<a,b<\infty$ and that $E({}^tuM\xi_1)<0$. Note then that the L\'evy process
\[Y_t:={}^tuMX_t=\sum_{k=0}^{N_t}{}^tuM\xi_k\,,\;\;\;t\ge0\,,\]
satisfies the conditions of Definition 1 in \cite{kp},  except that it is lattice. However, as noticed just after this 
definition, we can check from an analogous result in \cite{bd} that Theorem 1 in \cite{kp} is still valid in the lattice case. Fix an integer $k>0$ and let,
\[T=\inf\{t:Y_t< -k\}\,.\] Then Theorem 1 in \cite{kp} asserts that there exists a quasi-stationary distribution for $Y$, with respect to $T$. Let us 
denote by $\nu$ this distribution and let $\mu$ be a measure such that:
\begin{itemize}
\item[$(i)$] $\mu=\theta M^{-1}$, where  
$\theta=\theta_1\otimes\theta_2\otimes\dots\otimes\theta_d$ is a product probability measure on $\mathbb{Z}^d$ and $\theta M^{-1}$ is the image 
of $\theta$ by $M$, 
\item[$(ii)$] $\nu=\mu A^{-1}$, where  $A$ is the linear transformation, $Ax={}^tuMx$, $x\in\mathbb{Z}^d$.
\end{itemize}
Let $P_x$, $x\in\mathbb{Z}$ be the probability measure under which $Y$ starts from $x$, then from $(ii)$ and the quasi-stationarity of $\nu$, 
we obtain
\[\p_\mu(T>t)=P_\nu(T>t)=e^{-\alpha t}\,,\;\;\;\mbox{for some $\alpha>0$.}\] 
Observe that the time $T$ can be expressed as
\[T=\inf\{t:X_t\in D\}\,,\;\;\mbox{where}\;\;\;D=\{x\in \mathbb{Z}^d:{}^tuMx<-k\}\,. \]
Then let us show that $\mu$ is not quasi-stationary with respect to  time $T$. Let  $v=(v_1,\dots,v_d)\in\mathbb{Z}^d$ be another vector such that  $v_i=0$ if 
$u_i\neq 0$, for $i=1,\dots,d$ and set $Z:={}^tvMX$. Assume that $Z$ is not a degenerate process. Then by construction of $u$, $v$ and 
$\mu$,  the compound Poisson processes $Y$ and $Z$ are independent under $\p_\mu$. 
It follows that for  all $i\in\mathbb{Z}$ and $t\ge0$,
\[\p_\mu(Z_t=i\,|\,T>t)=\p_\mu(Z_t=i)\,.\]
If $\mu$ was quasi-stationary, then the last expression would be equal to $\mu B^{-1}(i)$, where $B:={}^tvM$,
hence $\mu B^{-1}$ would be a stationary distribution for the compound Poisson process $Z$. But such a distribution does not exist, as  is well 
known.

Note that this situation becomes trivial when the coordinates of the Poisson process $X=(X^{(1)},\dots,X^{(d)})$ are independent, that is 
$M=Id$. Let us chose $Y$ and $Z$ as follows,  $Y=X^{(1)}$ and $Z=X^{(2)}$. Then any measure $\mu$ of the form 
$\mu:=\nu\otimes\theta_2\otimes\dots\otimes\theta_d$,
where $\nu$ is the quasi-stationary distribution associated to $Y$ as above, is such that (\ref{exponential}) holds, although it is not quasi-stationary.\\

\noindent{\bf Remarks} {\it Contrary to the situations that are described just above, it may sometimes happens that exponentiality implies 
quasi-stationarity. Here are a couple of examples. 
\begin{itemize}
\item[$(i)$] Let $X$ be a birth and death process with birth rate $\lambda_n=n\lambda$ and death rate $\nu_n=n\nu$, when the process is 
in state $n$. In this case,  $S$ is the set $\{0,1,\dots\}$ of nonnegative integers.  Set $D=\{0\}$ and recall the definition of the first passage time,
\[T=\inf\{t:Z_t=0\}\,,\]
which is an {\it absorption time} in the present case. It is well known that, if $\nu>\lambda$, then $\p_k(T<\infty)=1$, 
for all $k\ge1$ and from the branching property, we have for all $k\in E=\{1,2,\dots\}$ and all $t>0$,
\begin{eqnarray}
\p_k(T\le t)=[\p_1(Z_t=0)]^k\,.\label{3261}
\end{eqnarray}
Le $q_t:=\p_1(Z_t=0)$ be the extinction probability, then from $(\ref{3261})$, for any probability measure $\mu$ on $E$,
the quantity $\p_\mu(Z_t=0)=\p_\mu(T\le t)$ corresponds to the generating function $G_\mu$ of $\mu$, evaluated at $q_t$, that is
\begin{eqnarray}
\p_\mu(T\le t)=G_\mu(q_t)\,.\label{4560}
\end{eqnarray}
In \cite{an}, p.$109$, we can find the  expression:
$q_t=\frac{\nu e^{(\nu-\lambda)t}-\nu}{\nu e^{(\nu-\lambda)t}-\lambda}$,
so that if $\mu_{\alpha}$ is a distribution which satisfies  $\p_{\mu_\alpha}(T> t)=e^{-\alpha t}$, for some $\alpha>0$, then
from $(\ref{4560})$, its generating function is given by:
%With $q_t^{-1}=\ln\left(\frac{\nu-\lambda t}{\nu(1-t)}\right)^{\frac1{\nu-\lambda}}$, we obtain,
\[G_{\mu_\alpha}(t)=1-\left(\frac{\nu-\lambda t}{\nu(1-t)}\right)^{\frac{-\alpha}{\nu-\lambda}}\,,\;\;\;t\in[0,1)\,.\]
This shows that for any $\alpha>0$ there is a unique distribution satisfying $\p_{\mu_\alpha}(T>t)=e^{-\alpha t}$. In the
case of continuous state branching processes, a similar expression for the Laplace transform of $\mu_\alpha$ has been obtained in
\cite{la1}, see p. 438 therein. 
\item[$(ii)$] In the case where $S$ is a finite set, another example where exponentiality implies quasi-stationarity is
given in part $(iii)$ of Proposition $2.1$ of \cite{ma}. The Markov chain that is considered in this work is a random walk in
the finite set $\{0,1,\dots,N\}$ that is killed at $0$.
\item[$(iii)$] In the case of continuous state space Markov processes, other examples where exponentiality implies quasi-stationarity
may be found in \cite{hr}. In this work it is proved that if the absorption time of a positive selfsimilar Markov process is
exponentially distributed under some initial distribution, then the latter is necessarily quasi-stationary.
\end{itemize}}

\subsection{Application to the emergence time of a mutant escaping treatment.}\label{ept1}

Let us consider the case of a pathogen population living on a host population. At each time $t$,  the whole host population
is either treated or not. A pathogen individual can mutate to defeat the treatment. We assume that each pathogen has the same mutation rate 
during a reproduction. Then, the probability that at least one pathogen mutates in the population is proportional to the pathogen population size. 
Since a treatment controls the pathogen population size, we assume that the latter takes two different values according to the 
presence or absence of the treatment. Thus, the mutant emergence rate takes two different values. In presence of the treatment, the pathogen 
population size is low, then the mutant emergence rate is low. In absence of treatment, the pathogen population size is high, then the mutant 
emergence rate is high.Then, the dynamics of the pathogen population size is described as a Markov chain $X$ whose state space $S$ is split up 
in three parts, that is $S=E_1\cup E_2\cup E_3$, with~:
\begin{itemize}
\item[.] $E_1$, the set of values of the pathogen population size when the population contains at least one mutant,
\item[.] $E_2$, the set of values of the pathogen population size when the population contains no mutants and its size is less than
a given value $K$,
\item[.] $E_3$, the set of values of the pathogen population size when the population contains no mutants and its size is greater than $K$.
\end{itemize}
The transition rates from $E_i$ to $E_j$, $1\le i,j\le 3$ are denoted by $\bar{q}_{ij}$, in accordance with the 
notation of  Theorem \ref{th3}. The set $E_2$ corresponds to the presence of treatment and in this case, the number of pathogens is low.  
The set $E_3$ corresponds to the absence of treatment and the number of pathogens is  high. In each case, the number of pathogens does not fluctuate 
very much, so that we can assume that the transition rates $\bar{q}_{23}$ and $\bar{q}_{32}$ 
between $E_2$ and $E_3$ are constant. They depend only on the treatment strategy, that is on the choice to use a treatment or not at time $t$. 
Then for the same reasons both mutant emergence rates $\bar{q}_{21}$ and $\bar{q}_{31}$ are supposed to be constant. From the present model, 
$\bar{q}_{21}$ should be much  lower than $\bar{q}_{31}$. The emergence time is then defined as $T= \inf \{t \geq 0 : X_t \in E_1\}$.\\

Let us consider a treatment strategy ensuring that $\mu(E_2)=\bar{\mu}_2$ is the probability for the pathogen population size to be less than $K$ before 
a mutation occurs. Similarly, the probability for the size to be greater than $K$ before mutation, is
$\mu(E_3)=\bar{\mu}_3=1-\bar{\mu}_2$.
From Theorem \ref{th3}, $T$ is exponentially distributed with parameter $\alpha>0$, if $\bar{\mu}$ solves the equation~:
\begin{equation}\label{lawqs}
\bar{\mu} \bar{Q}^T=-\alpha \bar{\mu}\,,
\end{equation}
with
\begin{equation*} \bar{Q}^T=\left( \begin{matrix}
-\bar{q}_{23}-\bar{q}_{21} &  \bar{q}_{23}   \\
\bar{q}_{32} & -\bar{q}_{32}-\bar{q}_{31} \\
\end{matrix} \right)\,.\end{equation*}
Let us set $\alpha=\bar{\mu}_2\bar{q}_{21} + \bar{\mu}_3\bar{q}_{31}$ and
\[\bar{\mu}_2=\frac{\bar{q}_{21}-\bar{q}_{31}+\bar{q}_{23}+\bar{q}_{32}-\sqrt{(\bar{q}_{21}-\bar{q}_{31}+\bar{q}_{23}-\bar{q}_{32})^2+
4\bar{q}_{23}\bar{q}_{32}}}{2(\bar{q}_{21}-\bar{q}_{31})}\,.\]
Then we can check that $\bar{\mu}=(\bar{\mu}_2,\bar{\mu}_3)$ is a solution of (\ref{lawqs}). Therefore, with this choice for $\alpha$ and
$\bar{\mu}$, the time $T$ is exponentially distributed with parameter $\alpha>0$.\\

From a biological point of view, these results may be interpreted as follows. The rate $\bar{\mu}_2$ represents the proportion
of time during which the host population has been treated. Then from this proportion of time, we can determine the distribution of the emergence 
time of a mutant pathogen.

\vspace*{.5in}
\noindent {\bf Acknowledgement} We are very grateful to Professor Servet Mart\'{\i}nez to have pointed out the reference \cite{ma} and
given access to its preliminary version. We also thank the referees for their constructive remarks, and especially one of them for the short proof
of Proposition \ref{prop1}.

%\vspace*{.7in}

 \end{document}